\documentclass[10pt,leqno]{amsart}

\usepackage[dvipsnames]{xcolor}
\definecolor{webgreen}{rgb}{0,.5,0}
\definecolor{webbrown}{rgb}{.6,0,0}

\usepackage{amssymb,amsmath,amsthm}
\usepackage{graphicx}
\usepackage{indentfirst}
\usepackage{csquotes}
\usepackage{paralist}
\usepackage{fancyhdr}
\usepackage{etoolbox}
\usepackage{multicol}
\usepackage{booktabs}   
\usepackage{seqsplit}   
\usepackage{tabularx}   
\usepackage{float}
\usepackage[colorlinks=true,
linkcolor=webgreen,
filecolor=webbrown,
citecolor=webgreen,
urlcolor=black]{hyperref}

\newcommand{\seqnum}[1]{\href{https://oeis.org/#1}{\textrm{\underline{#1}}}}

\definecolor{codegreen}{rgb}{0,0.6,0}
\definecolor{codegray}{rgb}{0.5,0.5,0.5}
\definecolor{codepurple}{rgb}{0.58,0,0.82}
\definecolor{backcolour}{rgb}{0.95,0.95,0.92}

\usepackage{upquote}
\usepackage{listings}
\theoremstyle{plain}
\newtheorem{theorem}{Theorem}[section] 

\newtheorem{corollary}[theorem]{Corollary}
\newtheorem{conjecture}[theorem]{Conjecture}

\theoremstyle{definition}
\newtheorem{definition}[theorem]{Definition}

\newtheorem{algorithm}{Algorithm}[section]

\theoremstyle{remark}
\newtheorem{remark}[theorem]{Remark}

\begin{document}

\title{Structured Solutions of Prime-Base Binomial Congruences}

\author[G. A. Guedes]{G. A. Guedes}
\address{Departamento de Matem\'atica, Universidade Federal Rural de Pernambuco, Recife, PE 52171-900, Brasil}
\email{gabriel.guedes@ufrpe.br}

\author[R. N. Machado Jr.]{R. N. Machado Jr.}
\address{Departamento de Matem\'atica, Universidade Federal Rural de Pernambuco, Recife, PE 52171-900, Brasil}
\email{ricardo.machadojunior@ufrpe.br}

\date{\today}


\begin{abstract}
In this paper, we study the congruence $\binom{qn}{n} \equiv q^n \pmod n$ for a prime base $q$. Motivated by the OEIS sequence \seqnum{A080469} and the conjectural existence of infinitely many ternary solutions of the form $n=3^t p$, we analyze the more general family $n=q^t p$, where $p\neq q$ is prime. Our main result shows that, in this family, the congruence is equivalent to two independent conditions: a congruence modulo $p$ and an inequality in the sum of the digits. This reduces the search for such solutions to factoring an explicit integer and applying a base-$q$ digit-sum filter. We use this criterion to produce new large solutions for $q\in\{2,3,5,7,11\}$. We also prove that square solutions $n=p^2$ are exactly governed by Wieferich primes in base $q$.
\end{abstract}

\maketitle

\section{Introduction}

Binomial congruences arising from Wolstenholme-type phenomena connect combinatorial identities, $p$-adic methods, and computational number theory; see, for example, Helou and Terjanian \cite{HelouTerjanian2008}, McIntosh \cite{McIntosh1995} and Me\v{s}trovi\'{c} \cite{Mestrovic2011}. These congruences have been extensively generalized in various directions within the recent literature, including the survey by Me\v{s}trovi\'{c} \cite{Mestrovic2014} and the Lucasnomial analogues studied by Ballot \cite{Ballot2020}. The primary focus of the present article is the binomial congruence for an arbitrary prime base $q$, given by
$$ \binom{qn}{n} \equiv q^n \pmod{n}. $$

Motivated by the OEIS conjecture \seqnum{A080469} regarding structured ternary solutions of the form $n=3^t p$, our approach generalizes the foundations established in our previous work \cite{GuedesMachado2023} on the binary analogue $\binom{2n}{n} \equiv 2^n \pmod{n}$. Our approach relies on reducing the global congruence into two independent conditions: a purely modular constraint and a base-$q$ digit-sum requirement. This ternary sequence belongs to a wider OEIS family attached to prime-base binomial congruences, including \seqnum{A084699} for $q=2$, \seqnum{A109760} for $q=5$, \seqnum{A109769} for $q=7$, and \seqnum{A392484} for $q=11$. This broader perspective motivates the study of
$\binom{qn}{n}\equiv q^n \pmod n$
for arbitrary prime bases $q$. We then apply the resulting criterion to obtain new large structured solutions for $q\in\{2,3,5,7,11\}.$

Furthermore, we extend the classical Wieferich connection observed in the base $2$ case to an arbitrary prime base $q$. We demonstrate that primes $p$ are Wieferich in base $q$ if and only if the square $n=p^2$ is a solution to $\binom{qn}{n} \equiv q^n \pmod{n}$. This offers a uniform characterization of these rare primes within the context of binomial congruences. A survey by Katz \cite{KatzWieferich} provides essential historical context for these Wieferich and Mirimanoff-type primes.

This paper is organized as follows. Section 2 recalls the notation and classical tools needed for the argument, including base-$q$ digit sums and standard congruences for binomial coefficients. Section 3 proves the main structural reduction for integers of the form $n=q^t p$. Section 4 applies this reduction to the conjectural ternary family arising from \seqnum{A080469}. Section 5 describes the computational search and reports the new solutions found. Section 6 proves the connection with Wieferich primes in base $q$. Finally, Section 7 discusses possible extensions and open directions.

\section{Notation and auxiliary congruences}

We collect the notation and standard congruences used throughout the paper. Let \(q\geq 2\) be a prime integer base. For a non-negative integer \(m\), we write \(s_q(m)\) for the sum of the digits of \(m\) in base \(q\). We use repeatedly the elementary identity $s_q(q^r m)=s_q(m)$
for every \(r\geq 0\).

We also recall Legendre's formula in digit-sum form. If \(q\) is prime, then
$$v_q(m!)=\frac{m-s_q(m)}{q-1}\quad\Longrightarrow\quad v_q\left(\binom{N}{M}\right)
=\frac{s_q(M)+s_q(N-M)-s_q(N)}{q-1}.$$
We shall also use Lucas' theorem in the following simple form. If \(p\neq q\) are primes, then
$$\binom{q^{t+1}p}{q^t p} \equiv \binom{q^{t+1}}{q^t} \pmod p$$
and with Fermat's theorem,
$$q^{q^t p}\equiv q^{q^t}\pmod p,$$
this provides the modular component of our main reduction.

Finally, in the study of square solutions we use the standard Babbage--Jacobsthal congruence: for primes \(p\geq 5\) and \(p\nmid q\),
$$\binom{q p^2}{p^2}\equiv\binom{q p}{p}\equiv q\pmod {p^2}.$$
These elementary tools are the only preliminary facts needed for the structural reductions below.
For an extensive historical and theoretical background on binomial coefficients modulo prime powers, Kummer's carry counting theorems, and Wolstenholme-type phenomena, the reader is referred to Davis and Webb~\cite{DavisWebb1990}, Granville~\cite{Granville1997} and Me\v{s}trovi\'{c}~\cite{Mestrovic2011}.

\section{The reduction for $n = q^t p$}

In this section, we generalize the reduction obtained in our previous work \cite{GuedesMachado2023} to an arbitrary prime base $q$. Specifically, we prove that for integers of the form $n = q^t p$, the binomial congruence $\binom{qn}{n} \equiv q^n \pmod{n}$ reduces to independent arithmetic conditions that are well-suited for a computational search.

\begin{theorem}[Prime-base reduction for $n=q^t p$]
\label{thm:base-q-qt-p}
Let $q$ be a fixed prime. Let $t\geq 1$ and let $p\neq q$ be a prime. Set
$$n=q^t p \quad\text{and}\quad A_t^{(q)}:=\binom{q^{t+1}}{q^t}-q^{q^t}.$$
Then the binomial congruence
\begin{equation}
\label{eq:q-base-congruence}
    \binom{qn}{n}\equiv q^n \pmod n
\end{equation}
holds if and only if the following two conditions hold simultaneously:

\begin{enumerate}
    \item[(i)] $p\mid A_t^{(q)}$, i.e.
    $\binom{q^{t+1}}{q^t}\equiv q^{q^t}\pmod p;$
    \item[(ii)] 
    $s_q((q-1)p)\geq (q-1)t,$ where $s_q(m)$ denotes the sum of the digits of $m$ in base $q$.
\end{enumerate}
\end{theorem}

\begin{proof}
Since $\gcd(q^t,p)=1$, the Chinese Remainder Theorem reduces the Congruence \eqref{eq:q-base-congruence} modulo $n=q^t p$ to its reductions modulo $p$ and modulo $q^t$.

Modulo $p$, Lucas' theorem and Fermat's little theorem give
$$
\binom{qn}{n}
=
\binom{q^{t+1}p}{q^t p}
\equiv
\binom{q^{t+1}}{q^t}
\pmod p
$$
and
$$
q^n=q^{q^t p}=(q^p)^{q^t}\equiv q^{q^t}\pmod p,
$$
since $p\neq q$. Hence the condition {\em(i)} is true.

It remains to treat the modulus $q^t$. Since $n=q^t p\ge t$, we have $q^t\mid q^n$. Thus the congruence modulo $q^t$ is equivalent to
$$
q^t\mid \binom{qn}{n},
\qquad\text{that is,}\qquad
v_q\left(\binom{qn}{n}\right)\ge t.
$$
Using Legendre's formula, we obtain
$$
 v_q\left(\binom{qn}{n}\right) = \frac{s_q((q-1)n)+s_q(n)-s_q(qn)}{q-1} = \frac{s_q((q-1)n)}{q-1}.
$$
Here $s_q(qn)=s_q(n)$, because multiplication by $q$ appends a zero in base $q$. Since $n=q^t p$, multiplication by $q^t$ also appends $t$ zeros, and so
$$
s_q((q-1)n)=s_q((q-1)q^t p)=s_q((q-1)p).
$$
Therefore
$$
q^t\mid \binom{qn}{n}
\iff
s_q((q-1)p)\ge (q-1)t,
$$
which is condition (ii). Combining the two congruences completes the proof.
\end{proof}

\begin{corollary} \label{ternary_version}
Let $t\ge 1$ and let $p\neq 3$ be prime. Set $n=3^t p$.
Then $\binom{3n}{n} \equiv 3^n \pmod{n}$ holds for $n$ if and only if both conditions hold:
\begin{enumerate}
\item [(i)] $p\mid A_t^{(3)}$, i.e.
$\binom{3^{t+1}}{3^t}\equiv 3^{3^t}\pmod p$;
\item [(ii)] $s_{3}(2p)\ge 2t$.
\end{enumerate}
\end{corollary}

\begin{remark}
Corollary \ref{ternary_version} is the direct analogue of the binary structural theorem by Guedes and Machado
\cite[Theorem 2.1]{GuedesMachado2023}. Here the modular component is encoded by $A_t^{(3)}$,
while the $3$-adic component becomes a ternary digital condition.
\end{remark}

\section{OEIS sequences and the ternary conjecture} 

The sequence \seqnum{A080469} in the On-Line Encyclopedia of Integer Sequences (OEIS) catalogues integers $n$ such that $\binom{3n}{n} \equiv 3^n \pmod{n}$. In 2015, M. F. Hasler noted that a significant subsequence of these terms, specifically $a(n)$ for $n \in \{2, 6, 7, 8, 9, 10\}$, follows the form $3^k p$ where $p$ is a prime. This observation led to the conjecture that there are infinitely many solutions of this structural form.

\begin{conjecture}
The equation $\binom{3n}{n} \equiv 3^n \pmod{n}$ has infinitely many solutions of the form
$n=3^k\cdot p$, where $k$ is a positive integer and $p$ is a prime number.
More generally, the equation $\binom{qn}{n} \equiv q^n \pmod{n}$ has infinitely many solutions
of the form $n=q^k\cdot p$, where $p\neq q$ are prime numbers and $k$ is a
positive integer.
\end{conjecture}

Theorem \ref{thm:base-q-qt-p} provides a natural framework for this conjecture. Indeed, for solutions of the form $n=q^t p$, the global congruence is governed by two independent conditions:
\begin{multicols}{2}
\begin{enumerate}
    \item[{\em(i)}]  $p \mid \left( \binom{q^{t+1}}{q^t} - q^{q^t} \right)$;
    \item[{\em(ii)}] $s_q((q-1)p) \geq (q-1)t$.
\end{enumerate}
\end{multicols}
Thus the conjectural infinitude of such solutions can be viewed as the expectation that, for infinitely many values of $t$, at least one prime divisor of $A_t^{(q)}$ also passes the corresponding base-$q$ digit-sum filter.

\subsection{Heuristic perspective}

We present a heuristic argument supporting the expectation that infinitely many solutions of the form $n=q^t p$ should occur. By Theorem \ref{thm:base-q-qt-p}, for each fixed $t$, the modular condition naturally selects prime candidates among the prime divisors of
$A_t^{(q)}$.
A solution is then obtained whenever one of these primes also satisfies the digital condition
$s_q((q-1)p)\ge (q-1)t$.

The growth of $A_t^{(q)}$ suggests that this mechanism should continue to produce large sets of prime candidates as $t$ increases. For large prime factors $p$, it is natural to expect that the base-$q$ expansion of $(q-1)p$ will usually have a digit sum comparable to that of a typical integer of similar size. From this point of view, the inequality $s_q((q-1)p)\ge (q-1)t$ should be seen not as an exceptional restriction, but as a filter that many sufficiently large candidates may reasonably pass.

In summary, the two conditions in Theorem \ref{thm:base-q-qt-p} complement each other: the modular condition yields structured prime candidates, while the base-$q$ digit-sum constraint filters out those with overly sparse $q$-adic expansions. The numerical examples for $q\in\{2,3,5,7,11\}$ presented in Section \ref{sec:comp_metho} align with this reduction, suggesting that solutions of the form $n=q^t p$ occur regularly across prime bases rather than by mere coincidence.

\section{Computations and new terms}\label{sec:comp_metho}

To compute new solutions for the binomial congruence $\binom{qn}{n} \equiv q^n \pmod{n}$, we use a search routine based on candidate generation and integer factorization. This algorithm systematically identifies structured solutions of the form $n = q^t p$ by checking the two conditions established in Theorem \ref{thm:base-q-qt-p}.

Our computational method consists of two main steps. First, we compute the constant $A_t^{(q)} = \binom{q^{t+1}}{q^t} - q^{q^t}$ using Algorithm \ref{algo_gera_binom_} for fixed primes $q$ and exponents $t$. Because these constants grow double-exponentially, we extract their prime factors using the \textit{Generic Integer Factorization} suite Alpern \cite{alpern2026}. This tool employs methods such as the Lenstra Elliptic Curve Method (ECM) and the Self-Initializing Quadratic Sieve (SIQS) to handle the large bit-lengths of $A_t^{(q)}$.

In the second stage, each prime factor $p$ obtained from the factorization is tested against the base-$q$ digit-sum condition
$$s_q((q-1)p) \ge (q-1)t$$
using Algorithm \ref{algo_condicao_2}. The primes that pass this test yield solutions of the form $n=q^t p$.

\subsection{Algorithmic implementation}

We implemented the following algorithms in SageMath version 9.5 to identify the new solutions reported in this work.

\begin{algorithm}\label{algo_gera_binom_}Generation of the modular constant
\begin{lstlisting}[language=Python]
def generate_modular_constant(base, t):
    # Calculates the constant Atq = binom(base^(t+1),base^t)-base^(base^t)
    # The prime factors of this constant are candidates satisfying the 
    # modular condition (i) of Theorem 1. Implementation in SageMath.
    
    # Calculate Atq using SageMath's arbitrary-precision binomial function
    # This value is later factored to find suitable primes p
    Atq = (binomial(base**(t+1), base**(t)) - base**(base**t))
    
    return Atq
\end{lstlisting}
\end{algorithm}

\begin{algorithm}\label{algo_condicao_2} Verification of the digit-sum condition
\begin{lstlisting}[language=Python]
def verify_condition_sum_digits(q, t, prime_list):
    # Iterates through prime factors p to identify those satisfying 
    # condition (ii) of Theorem 1: sq((q-1)p) >= (q-1)t.
    
    for p in prime_list:
        # Compute the sum of digits of (q-1)*p in base q
        # SageMath's .digits() method returns the representation in base q
        sdpb = sum(((q-1)*p).digits(base=q))  
        
        # Verify if the q-adic digital sum meets the filter threshold
        if (q-1)*t <= sdpb:
            # If both conditions are satisfied, n = q^t * p is a valid solution
            print('%d^%d x %d = %d' % (q, t, p, (q**t)*p))
\end{lstlisting}
\end{algorithm}

\subsection{Computational results}

Using the method described above, we identified several new solutions for the congruence $\binom{qn}{n} \equiv q^n \pmod{n}$ in the form $n = q^t p$. 

Due to the double-exponential growth of the modular constant $A_t^{(q)}$, complete prime factorization was only achievable up to certain threshold limits. Utilizing the Alpertron suite, we successfully obtained the complete factorization for the constants $A_8^{(2)}$, $A_9^{(2)}$; $A_1^{(3)}$ through $A_5^{(3)}$; $A_1^{(5)}$ through $A_3^{(5)}$; $A_1^{(7)}, A_2^{(7)}$;  $A_1^{(11)}$, $A_3^{(11)}$ and $A_4^{(11)}$. For higher exponents, the factorization process was partial. We extracted a subset of prime factors for $A_6^{(3)}, A_7^{(3)}, A_4^{(5)}, A_5^{(5)}, A_6^{(5)}$, $A_3^{(7)}$, and $A_2^{(11)}$ before halting the execution due to computational constraints. 

\subsubsection{Primality verification and certification of solutions}

We summarize in Table \ref{tab:solutions} the values of $q, t$, and $p$ for the newly discovered solutions $n = q^t p$ to the congruence $\binom{qn}{n} \equiv q^n \pmod{n}$. Due to typographical constraints, the full decimal expansions of the exceptionally large prime factors have been truncated. The complete data for these primes is publicly available in our supplementary repository \cite{MachadoRepo}.
All prime factors reported in Table \ref{tab:solutions}, with the exception of the 21,288-digit probable prime, have been formally certified. We generated Elliptic Curve Primality Proving (ECPP) certificates for these integers using the Primo software \cite{MartinPrimo} (version 4.3.3), establishing primality beyond probabilistic tests like Baillie-PSW. The corresponding certificate files are available in our supplementary repository. Generating an ECPP certificate for a 21,288-digit prime is a massive computational undertaking. Any future certification of this number, whether achieved by the authors or the broader community, will be incorporated into the project's repository.

\begin{table}[H]
\centering
\caption{New solutions found for the congruence $\binom{qn}{n} \equiv q^n \pmod{n}$ with $n = q^t p$.}
\label{tab:solutions}
\small
\begin{tabularx}{\textwidth}{l c X}
\toprule
\textbf{Base ($q$)} & \textbf{Exp. ($t$)} & \textbf{Prime Factor ($p$)} \\
\midrule
2 & 8 & 4668330087227038797331533869889967 \\
2 & 8 & 271306135934685499552087746732362630141195875293262717461431 \\
2 & 9 & $425624709\dots 486709303$ ($278$ digits) \\
\midrule
3 & 3 & 107473761120101287 \\
3 & 4 & \seqsplit{315659155112048243028962773892209111491033137977545619169062109} \\
3 & 5 & 515025243037 \\
3 & 5 & 12613038334680779163587 \\
3 & 5 & 439736513990836620306262926703 \\
3 & 5 & \seqsplit{5154068650396768469936193910848234649146117150229577408388340963310157399993408377777858579508512769864973988305404218204527729811} \\
3 & 7 & 55088152723 \\
\midrule
5 & 2 & 9837733830594324400829 \\
5 & 3 & 3634667136029 \\
5 & 3 & \seqsplit{522771425508853711297596476396775636799424934500055536164860162531443474822165842761095895895952821575457523} \\
5 & 4 & 15207220980187081 \\
5 & 5 & 12428536940833320277 \\
5 & 6 & 192304957643 \\
\midrule
7 & 2 & 3350427851562891657143 \\
7 & 2 & 753394760358680657053588449200400901 \\
7 & 3 & 10539114841 \\
7 & 3 & 585698773549 \\
\midrule
11 & 2 & 105745859096254746901536512503843 \\
11 & 3 & $111789408\dots332286279$ ($1\,935$ digits) \\
11 & 4 & $426696441\dots074843387$ ($21\,288$ digits) \\
\bottomrule
\end{tabularx}
\end{table}

\subsection{OEIS sequences and new certified terms}

The solutions reported in Table \ref{tab:solutions} have a direct interpretation in terms of several integer sequences in the OEIS. For a fixed prime base $q$, let
\[
\mathcal S_q=\left\{ n\in \mathbb N:\ n \text{ is composite and }
\binom{qn}{n}\equiv q^n \pmod n \right\}.
\]
The cases considered here correspond to the following OEIS entries:
$$
\begin{array}{c|c|c}
q & \text{OEIS sequence} & \text{defining congruence}\\
\hline
\rule{0pt}{15pt} 2 & \seqnum{A084699} & \binom{2n}{n}\equiv 2^n \pmod n\\[0.05in]
3 & \seqnum{A080469} & \binom{3n}{n}\equiv 3^n \pmod n\\[0.05in]
5 & \seqnum{A109760} & \binom{5n}{n}\equiv 5^n \pmod n\\[0.05in]
7 & \seqnum{A109769} & \binom{7n}{n}\equiv 7^n \pmod n\\[0.06in]
11 & \seqnum{A392484} & \binom{11n}{n}\equiv 11^n \pmod n.
\end{array}
$$

It is important to stress that Table \ref{tab:solutions} is not intended to provide consecutive new terms for these sequences. The OEIS lists are ordered by the magnitude of $n$, whereas our search is restricted exclusively to the highly structured subfamily $n=q^t p$ (where $p\ne q$ is prime). Thus, the certified entries in Table~\ref{tab:solutions} are rigorously established elements of the corresponding OEIS sequences. The entry involving the 21,288-digit probable prime should be regarded as a conditional structured candidate pending a formal primality certificate.

For $q=3$, these results are especially relevant. They provide substantial evidence for the conjectural phenomenon, originally observed by M. F. Hasler in \seqnum{A080469}, predicting that infinitely many terms should occur in the form $3^k p$. The new ternary solutions feature prime factors substantially larger than those visible from direct searches. Furthermore, the analogous examples for $q \in \{2, 5, 7, 11\}$ suggest that this mechanism extends beyond the ternary case, reflecting a broader prime-base phenomenon governed by the interaction between Lucas-type modular reduction and base-$q$ digit sums.

In practice, the certification of these terms highlights the utility of our reduction. By converting a computationally intractable global congruence into independent arithmetic conditions, we obtain a rigorous proof for solutions that lie far beyond the reach of direct computation.

\section{Square solutions and Wieferich primes} 

In our previous work, which focused on the binary congruence $\binom{2n}{n}\equiv 2^n \pmod n$, we observed that square solutions of the form $n=p^2$ are naturally connected with classical Wieferich primes. More precisely, the case $q=2$ leads to the usual Wieferich condition $2^{p-1}\equiv 1 \pmod {p^2}$.

In this section, we show that this connection is not restricted to the binary setting. The argument extends to an arbitrary prime base $q$, replacing classical Wieferich primes with Wieferich primes in base $q$ (see Dorais and Klyve \cite{Dorais2011} for extensive computational searches regarding these base-$q$ analogues).

\begin{definition}[Wieferich in base $q$]\label{def:wieferich}
Let $q\ge 2$ be an integer and $p\nmid q$ a prime.
We say $p$ is \emph{Wieferich in base $q$} if
$$q^{p-1}\equiv 1\pmod{p^2}.$$
\end{definition}

\begin{theorem}[Wieferich base $q$ $\Longleftrightarrow$ square solution]\label{thm:wieferich-general}
Let $q$ and $p$ be primes with $p\neq q$ and $p\ge 5$.
Then
$$
\binom{q p^2}{p^2}\equiv q^{p^2}\pmod{p^2}
\quad\Longleftrightarrow\quad
q^{p-1}\equiv 1\pmod{p^2}.
$$
Equivalently, $p$ is Wieferich in base $q$ if and only if $n=p^2$ is a solution of
$\binom{qn}{n}\equiv q^n\pmod n$ (with $n=p^2$).
\end{theorem}

\begin{proof}

A Babbage/Jacobsthal-type lifting for binomial coefficients modulo prime powers
(see, e.g., Granville \cite{Granville1997}) implies, for $p\ge 5$, that
$$
\binom{q p^2}{p^2}\equiv \binom{q p}{p}\pmod{p^2},
$$
and Babbage's congruence gives
$$
\binom{q p}{p}\equiv q\pmod{p^2}.
$$
Hence the left side is congruent to $q$ modulo $p^2$.
Therefore $\binom{q p^2}{p^2}\equiv q^{p^2}\pmod{p^2}$ is equivalent to $q\equiv q^{p^2}\pmod{p^2}$, i.e.
$q^{p^2-1}\equiv 1\pmod{p^2}$.
Since $q$ is a unit modulo $p^2$, its order divides $\varphi(p^2)=p(p-1)$.
If $q^{p^2-1}\equiv 1$, then the order of $q$  also divides $p^2-1=(p-1)(p+1)$, hence divides
$\gcd(p(p-1),(p-1)(p+1))=p-1$. Thus $q^{p-1}\equiv 1\pmod{p^2}$.
Conversely, $q^{p-1}\equiv 1 \pmod{p^2}$ implies $q^{p^2-1}=(q^{p-1})^{p+1}\equiv 1 \pmod{p^2}$.
\end{proof}

\begin{remark}
For $q=3$, Theorem \ref{thm:wieferich-general} explains the appearance of $n=11^2=121$
among the smallest composite solutions to $\binom{3n}{n} \equiv 3^n \pmod{n}$ (since $11$ is Wieferich in base $3$).
For historical background on Wieferich/Mirimanoff-type primes see Katz \cite{KatzWieferich};
for computational aspects of generalized Wieferich congruences, see also
Keller and Richstein~\cite{KellerRichstein2005}.
\end{remark}

\section{Acknowledgments}

We would like to thank the Department of Mathematics of the Federal Rural University of Pernambuco (UFRPE), for the excellent work environment, with the appreciation of the broad spectrum that academic life can have.

\end{document}